\documentclass[11pt]{amsart}
\usepackage[T1]{fontenc}
\usepackage{amsmath,amssymb,amsthm}
\usepackage[margin=3cm]{geometry}
\usepackage[hidelinks]{hyperref}
\numberwithin{equation}{section}
\newtheorem{theorem}{Theorem}[section]
\newtheorem{proposition}[theorem]{Proposition}
\newtheorem{lemma}[theorem]{Lemma}
\newtheorem{corollary}[theorem]{Corollary}
\newcommand{\R}{\mathbb R}
\newcommand{\cH}{\mathcal H}
\newcommand{\cC}{\mathcal C}
\newcommand{\cX}{\mathcal X}
\newcommand{\cY}{\mathcal Y}
\newcommand{\cQ}{\mathcal Q}
\newcommand{\supp}{\operatorname{supp}}
\newcommand{\dist}{\operatorname{dist}}
\newcommand{\genus}{\operatorname{genus}}
\newcommand{\Id}{\operatorname{Id}}
\newcommand{\spec}{\operatorname{spec}}
\newcommand{\eps}{\varepsilon}
\newcommand{\dd}{\,d}
\newcommand{\norm}[1]{\left\|#1\right\|}
\title{Genus Increase in Closed Mean Curvature Flows}
\author{Jingwen Chen}
\address{Department of Mathematics, Faculty of Science,
University of Macau,\newline
E11, Avenida da Universidade, Taipa,
Macau}
\email{jingwenchen@um.edu.mo}

\begin{document}
\begin{abstract}
\setlength{\emergencystretch}{1em}
We construct closed mean curvature flows with arbitrarily large genus increase in every smooth Riemannian three-manifold, giving counterexamples to Ilmanen's genus-reduction conjecture. We prove that a compact hypersurface with a single isolated conical singularity and a power rate of convergence to its cone admits a mean curvature flow whose forward rescaled limit is a prescribed expander in the tangent space at the singular point, smoothly asymptotic to that cone. Combining this theorem with closed mean curvature flows with prescribed tangent flows \cite{LZ,CLSZ} and the noncompact counterexample \cite{HMW} gives the desired examples.
\end{abstract}
\maketitle

\section{Introduction}\label{sec:statements}

In 1995, Ilmanen \cite{IlmanenLectures} conjectured that singularities of mean curvature flow in $\R^3$ should strictly reduce genus, except for neckpinches and shrinking spheres. Recently, Hoffman--Mart\'in--White \cite{HMW} constructed noncompact counterexamples by joining a fixed asymptotically conical shrinker to expanders of arbitrarily large genus smoothly asymptotic to the same cone. Whether such genus increase can occur for flows starting from smooth closed surfaces in $\R^3$ remained open.

We construct closed counterexamples to Ilmanen's conjecture in every smooth Riemannian three-manifold. The flow starts from one fixed smooth connected closed embedded surface, develops a single isolated conical singularity, and admits almost regular continuations with arbitrarily large genus increase. For each continuation, parabolic rescalings about the singularity converge to the corresponding Hoffman--Mart\'in--White flow for negative and positive times. 

The hypersurface dimension is $n\ge2$. We write $\cC\subset\R^{n+1}$ for a cone whose link $\cC\cap S^n$ is a smooth closed embedded hypersurface of $S^n$.

We first prove an existence theorem for mean curvature flows out of isolated conical singularities in $\R^{n+1}$. Assuming that the compact initial hypersurface $M_0$ is asymptotic to a cone $\cC$ at the vertex with the power rate specified in \eqref{eq:trace}, Theorem~\ref{thm:forward} constructs a closed mean curvature flow with any connected, properly embedded expander smoothly asymptotic to $\cC$ as its forward rescaled limit. Theorem~\ref{thm:twosided} joins this outgoing flow to a closed mean curvature flow with a prescribed shrinking tangent flow, constructed in the author's previous joint work with Tang-Kai Lee, Ao Sun, and Xinrui Zhao \cite{CLSZ} (see also \cite{LZ}), and computes the change of genus for surfaces.

A smooth family of hypersurfaces $M^n(t)$ in a Riemannian manifold $(N^{n+1},g)$, $t\in I$, moves by mean curvature if
\[
 (\partial_t x)^\perp=\vec H_g(x),
\]
where $\perp$ denotes the $g$-normal component of the velocity and $\vec H_g$ is the mean curvature vector.
A shrinker is a smooth embedded hypersurface $\Sigma^n\subset \R^{n+1}$ satisfying
\begin{equation*}
  \vec H_\Sigma+\frac{x^\perp}{2}=0.
\end{equation*}
Equivalently, the family of hypersurfaces $(\sqrt{-t}\Sigma)_{t<0}$ is a self-similarly shrinking solution to mean curvature flow. An expander is a smooth embedded hypersurface $E^n\subset \R^{n+1}$ satisfying
\begin{equation} \label{expander:eq}
 \vec H_E - \frac{x^\perp}{2} = 0.
\end{equation}
Equivalently, $E_t=\sqrt t\,E$, $t>0$, is a self-similarly expanding solution to mean curvature flow.

Huisken's monotonicity formula \cite{Huisken} relates singularity formation to shrinkers. The geometry and rigidity of asymptotically conical shrinking ends were studied by Lu Wang \cite{Wang}, and uniqueness of the corresponding multiplicity-one tangent flows was proved by Chodosh--Schulze \cite{CS}. Tang-Kai Lee and Xinrui Zhao \cite{LZ} realized asymptotically conical (AC) shrinkers as singularity models of closed mean curvature flows. In subsequent joint work with them and Ao Sun \cite{CLSZ}, the author extended this result to a broader class of shrinkers and to flows in Riemannian manifolds. These realization arguments use the Wa\.{z}ewski box method, following earlier work of Angenent--Vel\'azquez \cite{AV} and the construction by Stolarski \cite{Stolarski} in Ricci flow. 

To continue a flow through an isolated conical singularity, one must construct a flow out of the singular time-slice. Expanders provide local models for such flows: the expanding flow associated with an AC expander converges to its asymptotic cone smoothly away from the vertex as $t\downarrow0$; see, e.g., \cite{BWspace}. For $2\le n\le6$, Chodosh--Daniels-Holgate--Schulze \cite{CDHS} proved short-time smoothness of the outermost flows from isolated conical singularities and identified their forward rescaled limits with the outermost expanders. Angenent-Daskalopoulos-\v{S}e\v{s}um \cite{ADSsmoothing} constructed a smooth continuation of Vel\'azquez's mean curvature flow \cite{Velazquez} past an isolated Type II singularity in $\R^8$. The weighted analysis and deformation theory of AC expanders were developed by Bernstein-Wang \cite{BWspace,BWdegree}; Deruelle--Schulze \cite{DS} studied relative entropy and generic uniqueness for expanders with vanishing relative entropy.

Topology change along mean curvature flow has been actively studied over the past several decades. White \cite{WhiteTopology} proved that, for level set flows starting from smooth closed surfaces in $\R^3$, the first Betti number of the complement is non-increasing. In particular, the genus along such level set flows is non-increasing between regular times. For innermost and outermost flows from smooth closed surfaces in $\R^3$, Ilmanen's genus-reduction conjecture follows from \cite{CCS,BK}. Colding-Minicozzi-Pedersen \cite{CMP} discussed the relation between singularities of mean curvature flow and topology. In joint work with Pedro Gaspar \cite{ChenGaspar}, the author constructed eternal integral Brakke flows in $S^3$ connecting Clifford tori to equatorial spheres. The author later extended this construction to higher-dimensional spheres \cite{ChenSymmetric,ChenClifford}, connecting Clifford hypersurfaces to equatorial spheres. Adrian Chun-Pong Chu and Ao Sun \cite{ChuSun} studied topology change and used an interpolation argument to construct a genus-one shrinker with entropy lower than that of the Angenent torus. Recently, Diederen \cite{Diederen} constructed mean convex flows in $\R^{p+q}$ that pass through an isolated bubble-sheet singularity from a hypersurface diffeomorphic to $S^{p+q-1}$ to one diffeomorphic to $S^{p-1}\times S^q$.

Let $M_0\subset\R^{n+1}$ be compact, and suppose that $M_0\setminus\{0\}$ is a smooth embedded hypersurface in $\R^{n+1}\setminus\{0\}$. Let $\mathbf n_{\cC}$ be a unit normal on $\cC\setminus\{0\}$. We assume that, near the origin, $M_0$ is a normal graph over $\cC$ with height $h_{\cC}$ satisfying
\begin{equation}\label{eq:trace}
 q\longmapsto q+h_{\cC}(q)\mathbf{n}_{\cC}(q),\qquad
 |\nabla_{\cC}^j h_{\cC}(q)|\le C_j |q|^{1+\alpha_0-j}\quad(j\ge0)
\end{equation}
for some $\alpha_0>0$. This graph describes all of $M_0$ in $B_{r_*}(0)\setminus\{0\}$ for some $r_*>0$.

\begin{theorem}\label{thm:forward}
Let $M_0$ satisfy the assumptions above. Let $E\subset\R^{n+1}$ be a connected smooth complete properly embedded expander smoothly asymptotic with multiplicity one to the same cone $\cC$. There are $T_E>0$ and a family of smooth closed embedded hypersurfaces $M_t\subset\R^{n+1}$, $0<t\le T_E$, moving by mean curvature, such that
\begin{enumerate}
\item As $t \downarrow 0$, $M_t$ converges to $M_0$ in Hausdorff distance and as integral varifolds;
\item the flow extends smoothly to $t=0$ locally on $\R^{n+1}\setminus\{0\}$, with initial surface $M_0\setminus\{0\}$;
\item the rescaled flows converge smoothly on compact subsets of $\R^{n+1} \times (0,\infty)$:
\begin{equation}\label{eq:forwardmodel}
 \lambda^{-1}M_{\lambda^2t}\longrightarrow\sqrt{t}\,E
 \quad\hbox{in }C^\infty\quad(\lambda\downarrow0);
\end{equation}
\item for all small $R$ and $t$, the area satisfies
\begin{equation}\label{eq:areabound}
 \cH^n(M_t\cap B_R)\le C_{E,M_0}(R^n+t^{n/2}).
\end{equation}
\end{enumerate}
\end{theorem}

\begin{theorem}\label{thm:twosided}
Let $\Sigma\subset\R^{n+1}$ be a connected nonflat complete properly embedded smooth AC shrinker with cone $\cC$. Choose the closed mean curvature flow $M_t^-\subset\R^{n+1}$, $-1\le t<0$, constructed in \cite{CLSZ} and specified in Section~\ref{sec:incoming}. For each $E$ as in Theorem~\ref{thm:forward}, the incoming flow extends to an integral Brakke flow on $[-1,T_E]$, smooth except at $(0,0)$, whose positive-time hypersurfaces are closed and connected and satisfy \eqref{eq:forwardmodel}. Denote the positive time part of the continuation by $M_t^+$. If $n=2$, the associated spacetime track is an almost regular mean curvature flow in the sense of Bamler--Kleiner \cite{BK}, and
\begin{equation}\label{eq:genusjump}
 \genus(M_{t_+}^+)-\genus(M_{t_-}^-)=\genus(E)-\genus(\Sigma),
\end{equation}
for $t_-<0<t_+$ sufficiently close to $0$. 
\end{theorem}

Combining Theorem~\ref{thm:twosided} with the shrinkers in \cite{HMWshrink} and the expanders in \cite{HMW} gives closed counterexamples to Ilmanen's genus-reduction conjecture.

\begin{corollary}\label{cor:closed}
There is a smooth connected closed initial surface in $\R^3$ whose smooth MCF reaches a single isolated conical singularity and admits almost regular integral Brakke continuations with arbitrarily large genus increase.
\end{corollary}
\begin{proof}
By \cite[Theorem 13.1]{HMW}, there are an AC shrinker and expanders of arbitrarily large genus, all smoothly asymptotic to the same cone. Fix this shrinker and apply Theorem~\ref{thm:twosided} to each of these expanders. Formula \eqref{eq:genusjump} then gives arbitrarily large genus increase.
\end{proof}

Theorems~\ref{thm:forward} and~\ref{thm:twosided} extend to Riemannian ambient manifolds as follows.

\begin{theorem}\label{thm:riemannian}
Let $(N^{n+1},g)$ be a smooth Riemannian manifold and $\mathfrak p\in N$. Fix normal coordinates on a ball $U$ centered at $\mathfrak p$, identifying $(T_{\mathfrak p}N,g_{\mathfrak p})$ with $\R^{n+1}$. Suppose that $M_0\subset N$ is compact, $M_0\setminus\{\mathfrak p\}$ is a smooth embedded hypersurface without boundary, and its normal-coordinate image near $\mathfrak p$ satisfies \eqref{eq:trace}. For every connected, smooth, complete, properly embedded expander $E\subset T_{\mathfrak p}N$ smoothly asymptotic with multiplicity one to $\cC$, the conclusions of Theorem~\ref{thm:forward} hold with $\R^{n+1}$ and $0$ replaced by $N$ and $\mathfrak p$, respectively, and with \eqref{eq:forwardmodel} interpreted as
\begin{equation}\label{eq:riemannianmodel}
 \lambda^{-1}\exp_{\mathfrak p}^{-1}(M_{\lambda^2t}\cap U)
 \longrightarrow\sqrt t\,E
 \quad\hbox{locally smoothly for }t>0.
\end{equation}
The area estimate uses $\cH_g^n$ and $B_R^g(\mathfrak p)$, with a constant depending on $E,M_0,g$.

For a shrinker $\Sigma\subset T_{\mathfrak p}N$ as in Theorem~\ref{thm:twosided}, that theorem also holds for its Riemannian realization in \cite{CLSZ}, with singular point $(\mathfrak p,0)$ and continuation interval $[-T_-,T_E]$ for some $T_->0$. The incoming flow and $T_-$ can be fixed independently of $E$. For $n=2$, almost regularity is understood in the analogous Riemannian sense, with Definitions~2.4 and~2.6 of \cite{BK} formulated using geodesic balls and the metric $g$.
\end{theorem}

Applying Theorem~\ref{thm:riemannian} to the fixed shrinker and expanders of arbitrarily large genus in \cite{HMW} and using \eqref{eq:genusjump}, we obtain the following corollary.

\begin{corollary}\label{cor:riemannian}
Every smooth Riemannian three-manifold contains a smooth connected closed initial surface whose mean curvature flow develops a single isolated conical singularity and admits almost regular integral Brakke continuations with arbitrarily large genus increase.
\end{corollary}

\subsection{Proof ideas}

To prove Theorem~\ref{thm:forward}, we combine the normal graph functions of $\sqrt t\,E$ and $M_0$ over $\cC$, using a cutoff at radius comparable to $\sqrt{t}$ and another cutoff at a fixed radius, keeping the expander near the origin and $M_0$ outside a fixed ball. We then solve the MCF equation for a normal graph over the resulting hypersurfaces $N_t$, with graph height $o(\sqrt t)$; hence the forward rescaled limit is the prescribed expander.

The main step is to construct a right inverse for the linearized equation with a bound independent of the final time. Near the origin, the change of variables $\tau=\log t$ gives the operator $\partial_\tau-L_E$ on the fixed expander. Adding a compactly supported zeroth-order term gives a barrier for this equation. We remove that term using the discrete spectrum of $-L_E$, following \cite{CM,ChenAncient}. For a chosen exponent $\beta>0$, the components with eigenvalues greater than $-\beta$ are integrated from $-\infty$; the finitely many components with eigenvalues less than $-\beta$ are solved with a terminal condition. We combine this inverse with an inverse having zero initial data on a fixed closed hypersurface obtained by smoothing $M_0$ near the origin. Taking the time interval small compared with the square of the gluing radius makes the cutoff errors small. The contraction mapping theorem then gives the outgoing flow.

For Theorem~\ref{thm:twosided}, we take the incoming closed mean curvature flow from \cite{CLSZ}. Its weighted height estimate gives the power rate \eqref{eq:trace} for the time-zero limit. The incoming and outgoing flows converge to the same limit smoothly away from the origin and with convergence of total area. This allows us to join their weak flow identities across time zero. For surfaces, the parts of the incoming and outgoing surfaces outside a small ball about the origin are diffeomorphic. The change of genus is therefore $\genus(E)-\genus(\Sigma)$.

In the author's previous joint work \cite{CLSZ}, we used local coordinates to write mean curvature flow in a Riemannian manifold as a forced mean curvature flow in Euclidean space and estimated the forcing terms. We use the same approach to extend the results above to the Riemannian case.

We organize the paper as follows. In Section~\ref{sec:preliminary}, we recall the background and specify the weighted norms. In Section~\ref{sec:forward}, we prove Theorem~\ref{thm:forward}. In Section~\ref{sec:incoming}, we obtain \eqref{eq:trace} from the estimates in \cite{CLSZ}, and in Section~\ref{sec:weak}, we prove Theorem~\ref{thm:twosided}. In Section~\ref{sec:riemannian}, we treat the Riemannian case.

\subsection*{Acknowledgment}
The author is supported by the University of Macau Research Grant (MYRG-\allowbreak SRG2026-\allowbreak 00039-\allowbreak FSC).

\subsection*{AI declaration}

This work builds on the author's previous joint work \cite{CLSZ}. Before using AI, the author had developed the main ideas and proof strategy: constructing a flow out of a singularity with a prescribed expander to obtain genus increase across the singularity, gluing scaled copies of the expander to the initial hypersurface, and carrying out the perturbation argument. The use of spectral decomposition and the contraction mapping theorem was inspired by the work of Choi and Mantoulidis \cite{CM}, and the analysis of AC expanders was inspired by the work of Bernstein-Wang \cite{BWspace,BWdegree} and Letian Chen \cite{ChenAncient}. The extension to Riemannian manifolds was based on ideas from the author's previous joint work \cite{CLSZ}.

The author explained these ideas to ChatGPT 6 Pro and used it mainly to help construct the global right inverse by combining the inner and outer right inverses and estimating the resulting cutoff errors. ChatGPT was also used to assist with drafting the manuscript, literature searches, calculations, and checking the arguments for possible gaps. The author checked and revised these arguments and takes full responsibility for the mathematical content and final exposition.

\section{Preliminaries}\label{sec:preliminary}

\subsection{Asymptotically conical expanders}\label{sec:AC}\label{sec:preliminary-graphs}

Write $r=|x|$.

We say that a smooth embedded hypersurface $E\subset\R^{n+1}$ is smoothly asymptotically conical (AC) with cone $\cC$ if $\lambda E\to\cC$ smoothly with multiplicity one on compact subsets of $\R^{n+1}\setminus\{0\}$ as $\lambda\downarrow0$; see \cite{BWspace}. For $n=2$, the conical ends are annuli and are transverse to sufficiently large spheres; this remains true under a normal graph of height $v$ with $|v|/r+|\nabla v|$ sufficiently small. Thus cutting off these ends preserves the genus. We set $\genus(E)=\genus(E\cap\overline B_R)$ for all sufficiently large $R$, and similarly for $\Sigma$.

For the rest of this section, let $E$ be a connected, complete, properly embedded smooth AC expander with cone $\cC$. Expanders are stationary solutions of the rescaled mean curvature flow introduced below.

Let $F:M\times(0,T]\to\R^{n+1}$ be a smooth mean curvature flow. Set $\tau=\log t$ and define
\begin{equation*}
 \widetilde F(\cdot,\tau)=e^{-\tau/2}F(\cdot,e^\tau),
 \qquad -\infty<\tau\le\log T.
\end{equation*}
Then $\widetilde F$ satisfies the rescaled mean curvature flow equation
\begin{equation}\label{eq:rescaledMCF}
 (\partial_\tau\widetilde F)^\perp
 =\vec H_{\widetilde F}-\frac{\widetilde F^\perp}{2}.
\end{equation}

Write a nearby hypersurface as the normal graph $\{x+u(x)\mathbf n_E(x):x\in E\}$. Linearizing \eqref{eq:rescaledMCF} at $E$ gives $\partial_\tau u=L_Eu$, where the stability operator is
\begin{equation*}
 L_E=\Delta_E+\tfrac12 x\cdot\nabla_E+|A_E|^2-\tfrac12.
\end{equation*}

Following \cite{BWdegree}, for $l=0,2$, let $W^l_{1/4}(E)$ consist of functions in $W_{\mathrm{loc}}^{l,2}(E)$ with finite norm
\begin{equation*}
 \norm{f}_{l}^{W,(1/4)}=\left(\int_E\sum_{j=0}^{l}|\nabla_E^j f|^2e^{r^2/4}\dd\cH^n\right)^{1/2}.
\end{equation*}
By \cite[Lemma 4.1]{BWdegree}, $-L_E$ is self-adjoint in $W^0_{1/4}(E)$ with domain $W^2_{1/4}(E)$ and has a smooth orthonormal eigenbasis $\{\psi_i\}_{i=1}^\infty$ satisfying
\begin{equation*}
 -L_E\psi_i=\mu_i\psi_i,\qquad \mu_1<\mu_2\le\cdots\longrightarrow+\infty.
\end{equation*}
Only finitely many $\mu_i$ are negative. The estimate in the proof of this lemma, together with the closed graph theorem, shows that $\norm{f}_{2}^{W,(1/4)}$ is equivalent to $\norm{f}_{0}^{W,(1/4)}+\norm{L_Ef}_{0}^{W,(1/4)}$.

Choose the unit normal on $\cC$ to be the asymptotic limit of $\mathbf n_E$. For a fixed large $R_E$, write $E_t$ as the normal graph of $e(q,t)$ for $q\in\cC$ with $r=|q|>R_E\sqrt t$; its time derivatives are taken at fixed $q$. The proof of \cite[Proposition 3.3]{BWcompact}, in particular estimate (3.1) there, gives uniform local graphical bounds up to $t=0$ at every finite spatial order. The graphical MCF equation then gives smooth extension to $t=0$, with $e(\cdot,0)=0$. Parabolic scaling on annuli gives the mixed derivative estimates below; integrating $\partial_t e$ from zero gives the height estimates:
\begin{equation}\label{eq:eheight}
 |\nabla_{\cC}^j e|\le C_j t r^{-1-j},\qquad
 |\nabla_{\cC}^j\partial_t^m e|\le C_{jm}r^{1-j-2m},\quad m\ge1.
\end{equation}

\subsection{Weighted parabolic norms}\label{sec:weighted-norms}

Fix $0<\alpha<1$. We use the parabolic H\"older norms $C_P^{k,\alpha}$ of \cite{ChenAncient}.

Choose $0<a<\min\{\alpha_0,1\}$ and $\beta=a/2$ so that $-\beta\notin\spec(-L_E)$. Set $p=a-1$ and $\rho(x,t)=\sqrt{|t|+r^2}$.

The curvature may grow like $t^{-1/2}$, so we apply local Schauder estimates at spatial scale $\sqrt{t_0}$ and time scale $t_0$, following the argument for the expanding flow in \cite[proof of Lemma 5.4]{BWspace}. The normal graph construction and geometric estimates in Subsection~\ref{sec:reference} give uniform coordinate and coefficient bounds after this scaling.

Fix $0<\theta<1/8$ sufficiently small, independently of $r_0,T$. On $N_t$, constructed in Subsection~\ref{sec:reference}, we choose geodesic coordinates centered at $x_0$ on $N_{t_0}$ and extend them in time by a normal parametrization. Following the convention of \cite[equation (7.1)]{HS}, we define $Q_\theta(x_0,t_0)$ in these coordinates to be the backward parabolic cylinder $B_{\theta\sqrt{t_0}}(0)\times[t_0-\theta^2t_0,t_0]$. Its rescaled coordinate domain is $Q_\theta=B_\theta\times[-\theta^2,0]$. The same convention applies to $Q_{2\theta}$, to $E_t$, and to the fixed hypersurface used in Subsection~\ref{sec:outer}. Unless otherwise stated, time derivatives of functions on evolving hypersurfaces are taken with respect to normal parametrizations. For the cone graph function $e(q,t)$, however, $\partial_t$ is taken with $q$ fixed, as specified above. On these cylinders, $\rho$ is uniformly comparable to its value at $(x_0,t_0)$.

For the graph height $v$ and the given function $F$ in the linear equation, use the spaces $\cX_T,\cY_T$ with norms
\begin{align}
 \norm{v}_{\cX_T}
 &=\sup_{(x_0,t_0)}
  \frac{\norm{v}_{C_P^{2,\alpha}(Q_\theta(x_0,t_0))}}
    {t_0\rho(x_0,t_0)^p},\label{eq:Xnorm}\\
 \norm{F}_{\cY_T}
 &=\sup_{(x_0,t_0)}
  \frac{\norm{F}_{C_P^{0,\alpha}(Q_\theta(x_0,t_0))}}
    {\rho(x_0,t_0)^p}.\label{eq:Ynorm}
\end{align}
Here the local norms are the standard $C_P^{k,\alpha}$ norms in the rescaled coordinates on $Q_\theta$; the suprema are over $0<t_0\le T$ and $x_0\in N_{t_0}$. The spaces $\cX_T,\cY_T$ are Banach. Under these coordinate bounds, replacing $Q_\theta$ by $Q_{2\theta}$ gives equivalent norms. If $\norm v_{\cX_T}\le B$, then
\begin{equation}\label{eq:Xpoint}
 |v|\le CBt\rho^p,\quad
 |\nabla v|\le CB\sqrt t\rho^p,\quad
 |\nabla^2v|+|\partial_t v|\le CB\rho^p.
\end{equation}
All derivative estimates below include the corresponding scaled parabolic H\"older bounds in these coordinates, with geometric quantities written in local orthonormal frames. We use interior parabolic Schauder estimates \cite{CLSZ,Krylov}, combined with local boundedness estimates \cite{Lieberman} when the solution is controlled in $L^2$.

Let $u=u(x,\tau)$ and $F=F(x,\tau)$ be functions on the fixed expander $E$, for $-\infty<\tau\le\log T$. Under the change of variables $t=e^\tau$ and $y=\sqrt t\,x$, we apply \eqref{eq:Xnorm}--\eqref{eq:Ynorm} on $E_t$ to the transformed functions of $(y,t)$:
\begin{equation}\label{eq:nativenorms}
 \begin{aligned}
 \norm u_{\cX_E^\beta}
 &:=\norm{\sqrt t\,u(y/\sqrt t,\log t)}_{\cX_T},\\
 \norm F_{\cY_E^\beta}
 &:=\norm{t^{-1/2}F(y/\sqrt t,\log t)}_{\cY_T}.
 \end{aligned}
\end{equation}

\section{Flow from a conical hypersurface}\label{sec:forward}

\subsection{Construction of \texorpdfstring{$N_t$}{N\_t}}\label{sec:reference}

We join $E_t$ to $M_0$ by combining their normal graph heights over $\cC$ using two cutoffs. Choose smooth cutoffs $\chi,\zeta$ with $\chi=0$ on $[0,2R_E]$, $\chi=1$ on $[3R_E,\infty)$, $\zeta=1$ on $[0,5]$, and $\zeta=0$ on $[6,\infty)$. Take $r_0$ small with $8r_0<r_*$, and fix $\eps_*>0$ small enough that $3R_E\sqrt{\eps_*}<1/4$. For $0<t\le T\le\eps_*r_0^2$, use the following graph over $\cC$, with $r=|q|$:
\begin{equation*}
 q\longmapsto q+
 \left[\chi(r/\sqrt t)h_{\cC}(q)+\zeta(r/r_0)e(q,t)\right]
 \mathbf n_{\cC}(q),\qquad q\in\cC,\quad R_E\sqrt t<r<7r_0.
\end{equation*}
The displayed graph agrees with $E_t$ for $R_E\sqrt t<|q|\le2R_E\sqrt t$ and with $M_0$ for $6r_0\le|q|<7r_0$. Together with the unchanged portions of $E_t$ and $M_0$, it defines $N_t$.

Normal projection gives a graph height $h$ over $E_t$ on $r<4r_0$ and a graph height $k$ over $M_0$ on $r>r_0/2$; see \cite{BWspace}. Differentiating the cutoff formula and using \eqref{eq:trace} and \eqref{eq:eheight} gives
\begin{align}
 |\nabla^j\partial_t^m h|&\le C_{jm}\rho^{1+\alpha_0-j-2m},
    \qquad j,m\ge0,\label{eq:hjets}\\
 |\nabla^j k|&\le C_jt r^{-1-j},\label{eq:kjets1}\\
 |\nabla^j\partial_t^m k|&\le C_{jm}r^{1-j-2m},\qquad m\ge1.
    \label{eq:kjets2}
\end{align}
For small $r_0$ and $\eps_*$, these are small $C^1$ normal graphs at the corresponding scales. Thus $N_t$ is smooth, closed and embedded. Taking $r_0$ smaller if necessary gives $N_t\setminus B_{8r_0}=M_0\setminus B_{8r_0}$.

Let $V_N=\langle\partial_tX,\mathbf n_N\rangle$ be the scalar normal velocity of $N_t$, where $X(\cdot,t)$ is a smooth parametrization of $N_t$, and set $F_N=\langle\vec H_N,\mathbf n_N\rangle-V_N$. Since $E_t$ moves by mean curvature, the normal graph formulas in \cite{CDHS} and the preceding estimates give
\begin{align}
 |\nabla^j\partial_t^m A_N|+|\nabla^j\partial_t^m V_N|
  &\le C_{jm}\rho^{-1-j-2m},\label{eq:Ngeometry}\\
 |\nabla^j\partial_t^m F_N|
  &\le C_{jm}\rho^{\alpha_0-1-j-2m}\quad(r<4r_0),\label{eq:innerdefect}\\
 |\nabla^j\partial_t^m F_N|
  &\le C_{jm}r^{-1-j-2m}\quad(r>r_0/2).\label{eq:outerdefect}
\end{align}
The preceding graph estimates also give an embedding $(x,z)\mapsto x+z\mathbf n_{N_t}(x)$ for $x\in N_t$ and $|z|<c\rho(x,t)$; cf.\ \cite{BWspace}. Here $c>0$ and the constants above are independent of sufficiently small $r_0,T$.

\subsection{The linear equation on the expander}\label{sec:inner}

By \eqref{eq:Ngeometry} and the weighted norms \eqref{eq:Xnorm}--\eqref{eq:Ynorm}, $P_N=\partial_t-\Delta_N-|A_N|^2$ is bounded from $\cX_T$ to $\cY_T$, with norm independent of $r_0,T$. We first construct a right inverse for $P_E=\partial_t-\Delta_{E_t}-|A_{E_t}|^2$ on $E_t$, by rescaling to $\partial_\tau-L_E$ on $E$.

\begin{proposition}\label{prop:innerinverse}
For every $T>0$, the operator $\partial_\tau-L_E:\cX_E^\beta\to\cY_E^\beta$ has a bounded linear right inverse, with norm independent of $T$.
\end{proposition}
\begin{proof}
Let $F\in\cY_E^\beta$ and set $\phi(x)=(1+|x|^2)^{p/2}$. For $t=e^\tau$, we have $\rho(\sqrt t x,t)=\sqrt t(1+|x|^2)^{1/2}$. Since $(p+1)/2=\beta$, the definitions \eqref{eq:Ynorm} and \eqref{eq:nativenorms} give
\begin{equation*}
 |F(x,\tau)|\le\norm F_{\cY_E^\beta}\sqrt t\,\rho(\sqrt t x,t)^p=\norm F_{\cY_E^\beta}e^{\beta\tau}\phi(x).
\end{equation*}
For local estimates, we return to $P_E$ on $E_t$ and apply interior Schauder estimates on $Q_{2\theta}(y_0,t_0)$ to control solutions on $Q_\theta(y_0,t_0)$, following \cite{ChenAncient}. Here the cylinders are those defined in Subsection~\ref{sec:weighted-norms}, with $y_0\in E_{t_0}$. Parabolic dilation by $t_0^{-1/2}$ changes $E_t$ into $\sqrt{t/t_0}\,E$, with $1-4\theta^2\le t/t_0\le1$. The smoothness and AC asymptotics of $E$ give local graphical representations with uniform derivative bounds on this fixed time interval, so the Schauder constants are independent of $y_0,t_0$. Since $t\asymp t_0$ and $\rho(y,t)\asymp\rho(y_0,t_0)$ on $Q_{2\theta}(y_0,t_0)$, dividing the local estimates by the denominators in \eqref{eq:Xnorm}--\eqref{eq:Ynorm} and using \eqref{eq:nativenorms} gives the corresponding estimates on the fixed expander $E$.

Differentiation gives $\beta-L_E\phi/\phi=1+O(r^{-2})$, so this expression is at least $1/2$ outside a compact set. Choose a fixed nonnegative smooth compactly supported function $\kappa$ sufficiently large on that set so that
\begin{equation}\label{eq:killing}
 \beta-(L_E-\kappa)\phi/\phi\ge c_0>0.
\end{equation}
Standard linear parabolic theory \cite{Lieberman} gives weak solutions of $(\partial_\tau-L_E+\kappa)u=F$ with zero initial and lateral data on smooth bounded domains in $E$ and finite time intervals ending at $\log T$. The barriers $\pm C\norm F_{\cY_E^\beta}\,e^{\beta\tau}\phi$ bound these solutions independently of the domain and initial time. Interior Schauder estimates therefore give uniform local $C_P^{2,\alpha}$ bounds. As the domains exhaust $E$ and the initial time tends to $-\infty$, a diagonal subsequence converges locally, together with its first time derivative and spatial derivatives up to order two, to a classical solution $u_0$ of $(\partial_\tau-L_E+\kappa)u_0=F$ on $E\times(-\infty,\log T]$. Lower semicontinuity of the H\"older seminorm preserves the weighted estimates, so
\begin{equation}\label{eq:u0bound}
 \norm{u_0}_{\cX_E^\beta}\le C\norm F_{\cY_E^\beta}.
\end{equation}
If $u\in\cX_E^\beta$ satisfies $(\partial_\tau-L_E+\kappa)u=0$, then \eqref{eq:nativenorms} gives $|u(x,\tau)|\le\norm u_{\cX_E^\beta}e^{\beta\tau}\phi(x)$. Hence $w=e^{-\beta\tau}\phi^{-1}u$ is bounded and satisfies
\begin{equation*}
 w_\tau-\Delta_Ew-\tfrac12x\cdot\nabla_Ew-2\nabla_E\log\phi\cdot\nabla_Ew+cw=0,
\end{equation*}
where $c=\beta-(L_E-\kappa)\phi/\phi\ge c_0$ by \eqref{eq:killing}. With $q=1+r^2$,
\begin{equation*}
 (\Delta_E+(\tfrac12x^T+2\nabla_E\log\phi)\cdot\nabla_E)q=2n+r^2+4p|x^T|^2/(1+r^2)\le Cq.
\end{equation*}
Choose $K\ge0$ sufficiently large. For any $s<\log T$ and $\delta>0$, the function $\norm{w(s)}_\infty e^{-c_0(\tau-s)}+\delta e^{K(\tau-s)}q$ is a supersolution of the equation for $w$ on $[s,\log T]$, by the preceding inequality and $c\ge c_0$. It dominates both $w$ and $-w$ at $\tau=s$. Since $w$ is bounded and $q\to\infty$ at spatial infinity, it also dominates them on the lateral boundaries of sufficiently large bounded domains. Applying the maximum principle and letting these domains exhaust $E$ gives
\begin{equation*}
 |w(x,\tau)|\le\norm{w(s)}_\infty e^{-c_0(\tau-s)}+\delta e^{K(\tau-s)}q(x),\qquad s\le\tau\le\log T.
\end{equation*}
Letting $\delta\downarrow0$ and then $s\to-\infty$ gives $w=0$, since $w$ is bounded uniformly in space and time. Thus the solution $u_0$ constructed above is unique in $\cX_E^\beta$ and, by linearity of the equation, depends linearly on $F$.

Set $g=e^{-\beta\tau}\kappa u_0$. Since $\kappa$ has fixed compact support, the local parabolic $C_P^{2,\alpha}$ bound in \eqref{eq:u0bound} gives
\begin{equation}\label{eq:DAbound}
 \sup_{\tau\le\log T}\norm{g(\tau)}_{2}^{W,(1/4)}\le C\norm F_{\cY_E^\beta},
 \qquad g\in C( (-\infty,\log T];W^2_{1/4}(E)).
\end{equation}
It also gives a local $C_P^{0,\alpha}$ bound for $g$.

Using the orthonormal eigenfunctions fixed in Subsection~\ref{sec:AC}, define the $W^0_{1/4}(E)$-orthogonal projections by
\begin{equation*}
 \Pi_{<-\beta}f=\sum_{\mu_i<-\beta}\left(\int_E f\psi_i e^{r^2/4}\dd\cH^n\right)\psi_i,\qquad
 \Pi_{>-\beta}=\Id-\Pi_{<-\beta}.
\end{equation*}
The sum is finite because $\beta>0$ and the negative eigenspace of $-L_E$ is finite-dimensional. Since $-\beta\notin\spec(-L_E)$, the complementary projection is onto the closed span of the eigenfunctions with $\mu_i>-\beta$. Define
\begin{equation*}
 v(\tau)=\int_{-\infty}^{\tau}e^{(\tau-s)(L_E-\beta)}\Pi_{>-\beta}g(s)\dd s
    -\int_{\tau}^{\log T}e^{(\tau-s)(L_E-\beta)}\Pi_{<-\beta}g(s)\dd s.
\end{equation*}
This choice gives $\Pi_{<-\beta}v(\log T)=0$. This spectral construction follows \cite[Lemma 3.1]{CM} and \cite[Proposition 3.3]{ChenAncient}. Estimate \eqref{eq:DAbound}, the norm equivalence in Subsection~\ref{sec:AC}, and the spectral gap at $-\beta$ imply that both integrals converge in $W^2_{1/4}(E)$, and $v$ satisfies
\begin{equation*}
 \sup_{\tau\le\log T}\norm{v(\tau)}_2^{W,(1/4)}\le\frac{2C}{\dist(-\beta,\spec(-L_E))}\norm F_{\cY_E^\beta}.
\end{equation*}
Variation of constants gives $v_\tau-(L_E-\beta)v=g$ in $W^0_{1/4}(E)$, and local parabolic regularity makes $v$ classical. Thus $z=e^{\beta\tau}v$ satisfies $(\partial_\tau-L_E)z=\kappa u_0$, and $(\partial_\tau-L_E)(u_0+z)=F$.

Fix $(x_0,\tau_0)\in E\times(-\infty,\log T]$, and set $t_0=e^{\tau_0}$ and $y_0=\sqrt{t_0}\,x_0$. Consider $Q_{2\theta}(y_0,t_0)$, whose domain in the normal coordinates of Subsection~\ref{sec:weighted-norms} is $B_{2\theta\sqrt{t_0}}(0)\times[t_0-4\theta^2t_0,t_0]$. Pull this cylinder back to $E$ by $(y,t)\mapsto(y/\sqrt t,\log t)$. Along a normal parametrization of $E_t$, the rescaled position satisfies $dx/d\tau=-x^T/2$, so $d|x|^2/d\tau=-|x^T|^2$. Hence every point of the pulled-back cylinder has radius at least $|x_0|-2\theta$. In the coordinates of Subsection~\ref{sec:weighted-norms}, the coordinate measure on the rescaled domain $Q_{2\theta}=B_{2\theta}\times[-4\theta^2,0]$ is uniformly comparable to $\dd\cH^n\dd\tau$ on the pullback of $Q_{2\theta}(y_0,t_0)$ in $E$. The uniform $W^0_{1/4}(E)$ bound for $v$ therefore gives
\begin{equation}\label{eq:gaussiantail}
 \norm{v}_{L^2(Q_{2\theta})}
 \le C\norm F_{\cY_E^\beta}\exp[-(|x_0|-2\theta)^2/8]
\end{equation}
for large $|x_0|$. For such $x_0$, the pullback of $Q_{2\theta}(y_0,t_0)$ is disjoint from $\supp \kappa$, so $v_\tau-(L_E-\beta)v=0$ there. On $E_t$, the function $V(y,t)=v(y/\sqrt t,\log t)$ satisfies $(P_E+(\beta+\tfrac12)/t)V=0$ in the corresponding region. After parabolic dilation by $t_0^{-1/2}$, the additional zeroth-order coefficient becomes $(\beta+\tfrac12)t_0/t$, which is uniformly bounded on $Q_{2\theta}(y_0,t_0)$. Homogeneous interior estimates from $Q_{2\theta}(y_0,t_0)$ to $Q_\theta(y_0,t_0)$ and \eqref{eq:gaussiantail} give the same decay for the $C_P^{2,\alpha}$ norm. For bounded $|x_0|$, the estimate on $Q_\theta(y_0,t_0)$ follows from the local source bound and the $W^0_{1/4}(E)$ bound. Since the Gaussian decay is bounded by $C\phi(x_0)$, multiplication by $e^{\beta\tau}$ and \eqref{eq:nativenorms} give
\begin{equation*}
 \norm z_{\cX_E^\beta}\le C\norm F_{\cY_E^\beta},
\end{equation*}
with $C$ independent of $T$. Together with \eqref{eq:u0bound}, this shows that $F\mapsto u_0+z$ is a bounded linear right inverse of $\partial_\tau-L_E$.
\end{proof}

To return to $E_t$, set $v(\sqrt t x,t)=\sqrt t\,u(x,\tau)$. The rescaling gives
\begin{equation*}
 P_Ev=t^{-1/2}(\partial_\tau-L_E)u.
\end{equation*}
Thus $P_Ev=F$ is equivalent to $(\partial_\tau-L_E)u=\sqrt t\,F(\sqrt t x,t)$. By \eqref{eq:nativenorms}, Proposition~\ref{prop:innerinverse} gives a right inverse $R_i:\cY_T\to\cX_T$ for $P_E$, with norm $C_i$ independent of $T$.

\subsection{The linear equation on \texorpdfstring{$N_t$}{N\_t}}\label{sec:outer}\label{sec:parametrix}

Choose a fixed $\sigma>0$ with $\sigma^2\le\eps_*/24^2$. Apply the construction in Subsection~\ref{sec:reference} with gluing radius $r_0/24$ at time $\sigma^2r_0^2$ to obtain a closed embedded hypersurface $S_{r_0}$ agreeing with $M_0$ for $r\ge r_0/2$. Near the origin it agrees with $\sigma r_0E$. After dilation by $r_0^{-1}$, the contribution of $h_{\cC}$ has size $O(r_0^{\alpha_0})$ in all fixed derivative norms. Consequently
\begin{equation}\label{eq:Sgeometry}
 |\nabla^j A_{S_{r_0}}|\le C_j\widetilde\rho^{-1-j},
 \qquad \widetilde\rho(x)=\sqrt{r_0^2+r^2},
\end{equation}
and $(x,z)\mapsto x+z\mathbf n_{S_{r_0}}(x)$ is an embedding for $x\in S_{r_0}$ and $|z|<c\widetilde\rho(x)$. The constants may depend on $\sigma$ but are independent of $r_0$.

On the fixed $S_{r_0}$ set $P_S=\partial_t-\Delta_{S_{r_0}}-|A_{S_{r_0}}|^2$. Use \eqref{eq:Xnorm}--\eqref{eq:Ynorm} on $S_{r_0}$, with $\rho$ replaced by $\widetilde\rho$.

\begin{proposition}\label{prop:outerinverse}
There are $c_o,C_o>0$, independent of $r_0,T$, such that for $T\le c_o r_0^2$, $P_S$ has a bounded linear inverse $R_o$ with zero initial data and norm at most $C_o$.
\end{proposition}
\begin{proof}
On $S_{r_0}$, $|\nabla(\widetilde\rho^2)|^2\le4\widetilde\rho^2$ and $\Delta(\widetilde\rho^2)=2n+2\langle x,\vec H_{S_{r_0}}\rangle$, whose absolute value is bounded by \eqref{eq:Sgeometry}. Thus
\begin{equation*}
 |(\Delta_{S_{r_0}}+|A_{S_{r_0}}|^2)\widetilde\rho^p|
 \le C\widetilde\rho^{p-2},
 \qquad P_S(2t\widetilde\rho^p)\ge\widetilde\rho^p
 \quad(0<t\le c_o r_0^2).
\end{equation*}
For $\norm F_{\cY_T}\le1$, the semigroup integral
\begin{equation*}
 u(t)=\int_0^t e^{(t-s)(\Delta_{S_{r_0}}+|A_{S_{r_0}}|^2)}F(s)\dd s
\end{equation*}
solves $P_Su=F$ with zero initial data. Comparison gives $|u|\le2t\widetilde\rho^p$. On $Q_{2\theta}(x_0,t_0)$ on $S_{r_0}$, $\widetilde\rho$ is comparable to $\widetilde\rho(x_0)$ and $\sqrt t/\widetilde\rho\le\sqrt{T/r_0^2}\le\sqrt{c_o}$. By \eqref{eq:Sgeometry}, the coefficients of $P_S$ after rescaling are uniformly bounded. Interior Schauder estimates from $Q_{2\theta}(x_0,t_0)$ to $Q_\theta(x_0,t_0)$ give the asserted norm bound. Uniqueness follows from the maximum principle.
\end{proof}

Choose the following smooth cutoffs on $N_t$ depending on the ambient radius, and pull them back to $E_t$ and $S_{r_0}$:
\begin{center}
\begin{tabular}{c|c|c}
cutoff&equals one&equals zero\\ \hline
$\eta_i$&$r\le2r_0$&$r\ge3r_0$\\
$\eta_o=1-\eta_i$&$r\ge3r_0$&$r\le2r_0$\\
$\chi_i$&$r\le3r_0$&$r\ge(7/2)r_0$\\
$\chi_o$&$r\ge2r_0$&$r\le r_0$
\end{tabular}
\end{center}
Consequently
\begin{equation}\label{eq:partition}
 \chi_i\eta_i+\chi_o\eta_o=1
\end{equation}
on $N_t$.

Use the normal graph maps on $r<4r_0$ and $r>r_0/2$ to pull functions and differential operators back to $E_t$ and $S_{r_0}$, respectively. Since $\eta_i$ vanishes near $r=4r_0$ and $\eta_o$ vanishes near $r=r_0/2$, extend the pulled-back functions $\eta_iF$ and $\eta_oF$ by zero to $E_t$ and $S_{r_0}$. Define the approximate right inverse $Q$ by
\begin{equation*}
 QF=\chi_iR_i(\eta_iF)+\chi_oR_o(\eta_oF).
\end{equation*}

\begin{proposition}\label{prop:globalinverse}
After decreasing $c_o$ if necessary, there are constants $K_0,C$ independent of $r_0,T$ such that, for $T\le\min\{\eps_*,c_o,1\}r_0^2$, $\norm Q_{\cY_T\to\cX_T}\le K_0$ and, if $T=\eps r_0^2$,
\begin{equation}\label{eq:parametrixerror}
 \norm{P_NQ-\Id}_{\cY_T\to\cY_T}
 \le CC_i r_0^{\alpha_0}+C(C_i+C_o)(\sqrt\eps+\eps).
\end{equation}
If the right side is less than $1/2$, then
\begin{equation}\label{eq:globalR}
 R_T=Q(P_NQ)^{-1},\qquad
 P_NR_T=\Id,\qquad \norm{R_T}\le2K_0.
\end{equation}
\end{proposition}
\begin{proof}
Pullback by the normal graph maps gives equivalent norms on $\supp\chi_i$ and its image in $E_t$, and on $\supp\chi_o$ and its image in $S_{r_0}$. The latter agrees there with $M_0$. Choose $c_o$ sufficiently small so that, for $0<t_0\le T$, the cylinder $Q_{2\theta}(x_0,t_0)$ on $N_t$ lies in $\{r<4r_0\}$ if $x_0\in\supp\chi_i(\cdot,t_0)$, and in $\{r>r_0/2\}$ if $x_0\in\supp\chi_o(\cdot,t_0)$. The same inclusions hold for the cylinders $Q_{2\theta}$ on $E_t$ and $S_{r_0}$ centered at the corresponding preimages under the normal graph maps. After graph pullback, the difference between $P_N$ and $P_E$ on $r<4r_0$ has coefficients
\begin{equation}\label{eq:innercoeff}
 O(\rho^{\alpha_0})\nabla^2
 +O(\rho^{\alpha_0-1})\nabla
 +O(\rho^{\alpha_0-2}),
\end{equation}
and its difference from $P_S$ on $r>r_0/2$ has coefficients
\begin{equation}\label{eq:outercoeff}
 O(t/r^2)\nabla^2
 +O(t/r^3)\nabla
 +O(t/r^4).
\end{equation}
The indicated bounds also hold in $C_P^{0,\alpha}$ in the coordinates used for \eqref{eq:Xnorm}--\eqref{eq:Ynorm}.

On $Q_\theta(x_0,t_0)$, use the coordinates of Subsection~\ref{sec:weighted-norms} and put $\ell=\sqrt{t_0}$. By \eqref{eq:hjets}, $|h|/\rho+|\nabla h|\le C\rho^{\alpha_0}$ on $r<4r_0$. For the normal graph map $G(x)=x+h\mathbf n$, the spatial derivatives of order $m\ge1$ of $(G(x)-x)/\ell$ in these coordinates are bounded by $C\ell^{m-1}\rho^{1+\alpha_0-m}\le C\rho^{\alpha_0}$.

Along a normal parametrization of $E_t$, the graph map satisfies
\begin{equation*}
 \partial_tG=\vec H_{E_t}+(\partial_th)\mathbf n_{E_t}-h\nabla_{E_t}\langle\vec H_{E_t},\mathbf n_{E_t}\rangle.
\end{equation*}
For a function $u$ on $N_t$, the chain rule gives
\begin{equation*}
 \partial_t(u\circ G)=(\partial_tu)\circ G+\langle(\nabla_{N_t}u)\circ G,\partial_tG\rangle,
\end{equation*}
where the derivative on the left uses the normal parametrization of $E_t$, and $\partial_tu$ on the right uses that of $N_t$. Only the component of $\partial_tG$ tangent to $N_t$ contributes to the last term. The small $C^1$ graph bound, \eqref{eq:hjets}, and the geometric bounds for $E_t$ give a first-order coefficient $O(\rho^{\alpha_0-1})$. For the graph over the fixed $M_0$, $\partial_tG=(\partial_tk)\mathbf n_{M_0}$, whose tangential component has size $O(|\partial_tk||\nabla k|)=O(t/r^3)$ by \eqref{eq:kjets1}--\eqref{eq:kjets2}.

In the rescaled coordinates, these coefficients have size at most $C\ell\rho^{\alpha_0-1}$ and $C\ell t/r^3$, respectively. The differentiated graph estimates give the same bounds for their scaled H\"older norms. The ODEs defining the normal parametrizations and the chain rule give uniform $C_P^{2,\alpha}$ bounds for the changes of coordinates and their inverses. The denominators used in the norms are comparable at corresponding points.

The formulas for the metric and second fundamental form in \cite{CS}, together with their spatial derivatives, yield the second-, first-, and zeroth-order spatial differences in \eqref{eq:innercoeff}. All differentiated versions follow from \eqref{eq:hjets}; each spatial derivative in the rescaled coordinates introduces a factor at most $\ell/\rho$, and each time derivative introduces a factor at most $t/\rho^2$. On $r>r_0/2$, use \eqref{eq:kjets1}--\eqref{eq:kjets2} instead. In particular $t\partial_t(t/r^2)=O(t/r^2)$ on these cylinders, so the small factor in the outer coefficient H\"older norm is retained.

These norm comparisons, the bounds for $R_i$ and $R_o$, and the cutoff derivative bounds give the asserted bound for $Q$. Put $u_i=R_i(\eta_iF)$ and $u_o=R_o(\eta_oF)$. By \eqref{eq:partition},
\begin{equation*}
 P_NQF-F=\chi_i(P_N-P_E)u_i+\chi_o(P_N-P_S)u_o+[P_N,\chi_i]u_i+[P_N,\chi_o]u_o.
\end{equation*}
The inner coefficient differences in \eqref{eq:innercoeff}, applied to \eqref{eq:Xpoint}, give $CC_i r_0^{\alpha_0}\norm F_{\cY_T}$. For example, dividing the three terms by $\rho^p$ gives bounds $C_i\rho^{\alpha_0}$, $C_i\sqrt t\rho^{\alpha_0-1}$, and $C_it\rho^{\alpha_0-2}$, respectively; all are at most $CC_i r_0^{\alpha_0}$. The same product calculation is valid for the scaled H\"older seminorms. The coefficients in \eqref{eq:outercoeff} give $CC_o\eps\norm F_{\cY_T}$, since $t/r^2\le C\eps$ on that support.

The commutator is
\begin{equation}\label{eq:commutator}
 [P_N,\chi]u=(\partial_t\chi-\Delta_N\chi)u
             -2\nabla\chi\cdot\nabla u.
\end{equation}
On every transition annulus,
\begin{equation*}
 |\nabla\chi|\le C/r_0,\qquad
 |\partial_t\chi-\Delta_N\chi|\le C/r_0^2,
\end{equation*}
with the same bounds for the $C_P^{0,\alpha}$ norms in these coordinates. Since the cutoffs depend on the ambient radius, $\partial_t\chi=\chi'(r/r_0)V_N\langle x/r,\mathbf{n}_N\rangle/r_0$. Substitution of \eqref{eq:Xpoint} bounds \eqref{eq:commutator} by $C(C_i+C_o)(\sqrt T/r_0+T/r_0^2)\norm F_{\cY_T}$ in $\cY_T$. Multiplication by $\eta_i$ or $\eta_o$ is bounded, since its derivatives in the rescaled coordinates involve $\sqrt t/r_0$ and $t/r_0^2$. This proves \eqref{eq:parametrixerror}. The Neumann series gives \eqref{eq:globalR}.
\end{proof}

\subsection{The nonlinear equation}\label{sec:nonlinear}

\begin{lemma}\label{lem:nonlinear}
If $BT^{a/2}$ is sufficiently small and $\norm v_{\cX_T}\le B$, the normal graph $M_t=\{x+v(x,t)\mathbf n_{N_t}(x):x\in N_t\}$ is embedded and its MCF equation is
\begin{equation}\label{eq:graphPDE}
 P_Nv=F_N+\cQ_N(v).
\end{equation}
The nonlinearity vanishes to second order at zero and
\begin{equation}\label{eq:nonlinearLip}
 \norm{\cQ_N(v)-\cQ_N(w)}_{\cY_T}
 \le C_NBT^{a/2}\norm{v-w}_{\cX_T}
\end{equation}
on that ball, with uniform $C_N$.
\end{lemma}
\begin{proof}
Choose $X(\cdot,t)$ so that $\partial_tX=V_N\mathbf n_N$. Then $\partial_t\mathbf n_N=-\nabla_{N_t}V_N$, and differentiating the graph parametrization $X+v\mathbf n_N$ gives the velocity $(V_N+\partial_t v)\mathbf n_N-v\nabla_{N_t}V_N$. Taking its component along $\mathbf n_{M_t}$ gives
\begin{equation*}
 \partial_t v=
 \frac{\langle\vec H_{M_t},\mathbf n_{M_t}\rangle
    +v\langle\nabla V_N,\mathbf n_{M_t}\rangle}
   {\langle\mathbf n_N,\mathbf n_{M_t}\rangle}-V_N.
\end{equation*}
The normal graph formulas in \cite{CDHS} give $\Delta_N+|A_N|^2$ as the linearization of the mean curvature term. The denominator differs from one by $O(|\nabla v|^2)$, and the velocity correction is quadratic. Hence the linear part is $P_N$. On $Q_\theta(x_0,t_0)$, use the coordinates of Subsection~\ref{sec:weighted-norms} and put $z=v/\ell$, $\ell=\sqrt{t_0}$. By \eqref{eq:Ngeometry}, $\ell A_N$, $\ell V_N$, $\ell^2\nabla A_N$, and $\ell^2\nabla V_N$ have uniformly bounded H\"older norms in these coordinates. Furthermore
\begin{equation*}
 \norm z_{C_P^{2,\alpha}}
 \le CB\ell\rho(x_0,t_0)^p\le CBt_0^{a/2}.
\end{equation*}
The last inequality uses $p=a-1<0$ and $\rho(x_0,t_0)\ge\ell$. In these variables, the right side of the graph equation is a smooth function of $z,\nabla z,\nabla^2z$ and the geometric quantities just listed, linear in $\nabla^2z$. Its remainder after the linearization is quadratic and is locally Lipschitz from $C_P^{2,\alpha}$ to $C_P^{0,\alpha}$ with constant $CB\ell\rho(x_0,t_0)^p$, by the H\"older product estimate and smooth matrix inversion. The difference of the rescaled heights satisfies $\norm{(v-w)/\ell}_{C_P^{2,\alpha}}\le C\ell\rho(x_0,t_0)^p\norm{v-w}_{\cX_T}$. The original nonlinear remainder is $\ell^{-1}$ times the rescaled remainder. Thus rescaling and dividing by $\rho(x_0,t_0)^p$ gives $CB\ell\rho(x_0,t_0)^p\norm{v-w}_{\cX_T}\le CBt_0^{a/2}\norm{v-w}_{\cX_T}$. Taking the supremum proves \eqref{eq:nonlinearLip}. Also $|v|/\rho+|\nabla v|\le CBT^{a/2}$. For sufficiently small $BT^{a/2}$, the graph lies in the tubular neighborhood established above and is therefore embedded.
\end{proof}

\begin{proof}[Proof of Theorem~\ref{thm:forward}]
From \eqref{eq:innerdefect}--\eqref{eq:outerdefect} and \eqref{eq:Ynorm},
\begin{equation*}
 \norm{F_N}_{\cY_T}\le C_Fr_0^{-a}.
\end{equation*}
Indeed, after division by $\rho^p$, \eqref{eq:innerdefect} gives $O(\rho^{\alpha_0-a})$, and \eqref{eq:outerdefect} gives $O(r^{-a})$.

The constants $C_i,C_o,K_0,C_F,C_N$ may depend on the fixed geometric data, $\alpha,a,\theta,\sigma$, and the spectral gap, but are uniform in $r_0,T,\eps$ in the stated range. Put $K=2K_0$ and choose $B_0\ge2KC_F+1$. Next choose $0<\eps\le\min\{\eps_*,c_o,1\}$ so that $CB_0\eps^{a/2}$ is sufficiently small for Lemma~\ref{lem:nonlinear} and
\begin{equation*}
 C(C_i+C_o)(\sqrt\eps+\eps)<1/4,\qquad
 KC_NB_0\eps^{a/2}<1/4.
\end{equation*}
Finally choose $r_0>0$ small enough for the conical neighborhood and $CC_i r_0^{\alpha_0}<1/4$. Put $T=\eps r_0^2$ and $B=B_0r_0^{-a}$.
The right side of \eqref{eq:parametrixerror} is less than $1/2$, so $\norm{R_T}\le K$, and $BT^{a/2}=B_0\eps^{a/2}$. For $\norm v_{\cX_T}\le B$,
\begin{equation*}
 \norm{R_T\{F_N+\cQ_N(v)\}}_{\cX_T}
 \le KC_Fr_0^{-a}+KC_NB^2T^{a/2}\le\tfrac34B.
\end{equation*}
The map $v\mapsto R_T\{F_N+\cQ_N(v)\}$ therefore preserves the closed ball of radius $B$, with Lipschitz constant at most $KC_NBT^{a/2}<1/4$. Its fixed point solves \eqref{eq:graphPDE}. The resulting hypersurfaces are closed and embedded, and parabolic regularity makes the flow smooth for $t>0$.

On each compact subset away from the origin, the cutoff formula and \eqref{eq:eheight} give smooth convergence of $N_t$ to $M_0$. Together with \eqref{eq:Xpoint}, this shows that the local graph functions for $M_t$ converge to those for $M_0$ in $C^1$, with uniform Hessian bounds. Write the flow over a fixed plane as $u(y,t)$, with smooth initial trace $u_0$. It satisfies \cite{EH}
\begin{equation}\label{eq:planeMCF}
 u_t=a^{ij}(Du)u_{ij},\qquad
 a^{ij}(p)=\delta_{ij}-p_ip_j/(1+|p|^2).
\end{equation}
The Hessian bound gives $|u_t|\le C$, and spatial finite differences then give $|Du(y,t)-Du(y,s)|\le C|t-s|^{1/2}$ on a smaller patch. Thus the matrix $(a^{ij}(Du))$ is uniformly elliptic and its coefficients belong to $C_P^{0,\alpha}$ up to $t=0$. Local initial Schauder estimates \cite{Krylov}, applied to $u-u_0$, give $C_P^{2,\alpha}$ regularity up to zero, and repeated parabolic regularity gives smooth extension. All time derivatives are determined recursively by \eqref{eq:planeMCF} and $u_0$, so any incoming flow extending smoothly to this trace joins smoothly on the patch.

For any sufficiently large fixed $R$, the graph over $\sqrt t\,(E\cap B_R)$ lies in a ball of radius $O(\sqrt t)$. The normal correction is uniformly at most $CBt^{(1+a)/2}$. This proves the Hausdorff limit.

Under the parabolic rescaling $(y,t)\mapsto(\lambda y,\lambda^2t)$, \eqref{eq:hjets} shows that the height $\lambda^{-1}h(\lambda y,\lambda^2t)$ and its derivatives are $O(\lambda^{\alpha_0})$ on compact space-time subsets with $t>0$. The correction $v$ satisfies, on the same subsets,
\begin{equation*}
 \norm{\lambda^{-1}v(\lambda y,\lambda^2t)}_{C_P^{2,\alpha}}
 \le CB\lambda^a.
\end{equation*}
Thus the rescaled surfaces converge to $\sqrt{t}\,E$ in that norm. The graphical MCF equation and repeated interior parabolic regularity give the $C^\infty$ convergence in \eqref{eq:forwardmodel}.

The normal graph maps defining $N_t$ and $M_t$ have bounded Jacobians. For the portion of $M_t\cap B_R$ with base points on $N_t\cap B_{4r_0}$, its preimage in $E_t$ lies in $\sqrt t\,(E\cap B_{C(1+R/\sqrt t)})$. Smooth AC area growth gives
\begin{equation*}
 t^{n/2}\cH^n(E\cap B_{C(1+R/\sqrt t)})
 \le C(R^n+t^{n/2}).
\end{equation*}
For the portion of $M_t$ with base points on $N_t\setminus B_{4r_0}$, its graph representation over $M_0$ gives the same area bound, using \eqref{eq:trace} and the bounded graph Jacobians. This proves \eqref{eq:areabound}. This estimate, smooth convergence off zero, and common compact support imply varifold convergence and convergence of total area.
\end{proof}

\section{The incoming closed mean curvature flow and its limit}\label{sec:incoming}

Let $\Sigma$ be as in Theorem~\ref{thm:twosided}, and let $M_t^-$, $-1\le t<0$, be the closed mean curvature flow from the Euclidean construction of \cite[Theorem 1.2 and Proposition 5.17]{CLSZ} with zero forcing. We verify that its time-zero limit $M_0$ satisfies the hypotheses of Theorem~\ref{thm:forward}, in particular the power estimate \eqref{eq:trace}.

We retain the notation of \cite{CLSZ} and write $M_t=M_t^-$ in this section. Let $\widehat M_t$ denote the transplanted flow, $F$ the shrinker embedding, and $h$ the rescaled normal height of $\widehat M_t$. Put $\tau=-\log(-t)$ and let $\widetilde r\ge1$ be smooth, with $\widetilde r(x)=|F(x)|+1$ for $|F(x)|\ge1$. The transplanted flow is given by
\begin{equation}\label{eq:incominggraph}
 (-t)^{-1/2}\widehat M_t=\{F(x)+h(x,\tau)\mathbf n_F(x):x\in\Sigma\},
\end{equation}
and the pointwise bounds in \cite{CLSZ} give
\begin{equation}\label{eq:incomingbox}
 |h|\le\eps_0\widetilde r,\quad
 |\nabla h|\le\eps_1,\quad |\nabla^2h|\le\eps_2.
\end{equation}

The transplantation cutoff in \cite[Section 5]{CLSZ} equals one on a fixed inner region. The exterior velocity bounds in the proof of \cite[Proposition 5.7]{CLSZ} keep its complement outside a smaller ball for late times. For these times, the ball also contains the scaled graph of a fixed connected compact truncation of $\Sigma$. The small weighted $C^1$ bounds in \eqref{eq:incomingbox} preserve radial transversality beyond this truncation, so the portion of the transplanted graph in the ball is connected. Every point of $M_t$ in this ball comes from the region where the transplantation cutoff equals one. Thus the intersection of $M_t$ with the ball is a nonempty relatively open and closed subset of this connected portion of $\widehat M_t$, so the two flows agree in the ball. The estimates for $\widehat M_t$ therefore apply there to $M_t$. Together with the uniform exterior estimates in \cite[equation (2.14), Proposition 2.15 and Appendix B.1]{CLSZ}, these bounds give $\sup_{M_t}|A|\le C(-t)^{-1/2}+C$. The mean curvature is therefore integrable in time, so the flow parametrizations converge uniformly and $M_t$ has a compact Hausdorff limit $M_0$.

Let $\lambda_*\in(0,1/2)$ be the spectral parameter of the construction. By \cite[Theorem 4.13 and the proof of Proposition 5.7]{CLSZ}, there is $B>0$ such that, for all $\tau\ge0$,
\begin{align}
 |\nabla^\ell h|&\le B\widetilde r^{1+2\lambda_*-\ell}
      e^{-\lambda_*\tau},\qquad\ell=0,1,
      \label{eq:incomingpower}\\
 |\nabla^2h|&\le B\widetilde r^{2\lambda_*}
      e^{-\lambda_*\tau}.
      \label{eq:incominghessian}
\end{align}

Put $\Sigma_t=\sqrt{-t}\Sigma$ and $\alpha_0=2\lambda_*$. In a small annulus $r\asymp R$ and for $-cR^2<t<0$, equation~\eqref{eq:incomingpower} gives a normal height $k^-$ over $\Sigma_t$ satisfying
\begin{equation*}
 |k^-|\le CR^{1+\alpha_0},\qquad |\nabla_{\Sigma_t}k^-|\le CR^{\alpha_0}.
\end{equation*}
At corresponding points under dilation from $\Sigma$ to $\Sigma_t$, $k^-=\sqrt{-t}\,h$ and $|\nabla_{\Sigma_t}k^-|=|\nabla_\Sigma h|$. Substituting $\widetilde r\asymp R/\sqrt{-t}$ and $e^{-\lambda_*\tau}=(-t)^{\lambda_*}$ into \eqref{eq:incomingpower} cancels the time factors in both estimates.

For sufficiently small $R$ and $c$, $M_t$ and $\Sigma_t$ can locally be written as graphs over common fixed planar patches in $B_{2R}\setminus\overline B_{R/2}$ for $-cR^2<t<0$. Interior estimates \cite{EH} give uniform scaled derivative bounds on the smaller annulus $B_{3R/2}\setminus\overline B_{3R/4}$ for $-cR^2/2<t<0$. The functions $u,v$ representing the two flows as graphs over the same plane satisfy \eqref{eq:planeMCF}. Subtracting the equations satisfied by $u$ and $v$, we find that $w=u-v$ satisfies a homogeneous linear parabolic equation with uniform ellipticity and uniformly controlled coefficients in the rescaled coordinates. Since $\norm w_\infty\le CR^{1+\alpha_0}$, linear parabolic interior estimates \cite{Lieberman} give, on smaller patches,
\begin{equation*}
 |D^j\partial_t^m w|\le C_{jm}R^{1+\alpha_0-j-2m}.
\end{equation*}
These estimates also give smooth limits of the closed mean curvature flow on each fixed annulus. The smooth AC convergence in \cite{CLSZ} and the graphical MCF equation give a smooth extension of $\Sigma_t$ to $\cC$ at $t=0$ on these annuli. Letting $t\uparrow0$ and regraphing over $\cC$ proves all the bounds \eqref{eq:trace}. Overlapping graphs agree by uniqueness of normal projection to $\cC$.

The uniform exterior curvature and derivative bounds in \cite{CLSZ}, together with interior parabolic regularity, extend this smooth immersed convergence to every nonzero point of $M_0$.

There is no self-contact at time zero away from the origin. Transverse sheets at the limit would already intersect before zero. For tangent sheets, the preceding immersed convergence gives two graphs over a fixed ball for $t_1\le t<0$, with uniform derivative bounds. Embeddedness strictly orders them, so their difference $w$ is positive and satisfies a homogeneous uniformly parabolic linear equation with bounded coefficients. If their values at time zero agree at an interior point $y_0$, fix $t_a\in(t_1,0)$ sufficiently late. The interior Harnack inequality \cite{Lieberman} on a smaller fixed cylinder gives
\begin{equation*}
 0<w(y_0,t_a)\le Cw(y_0,t),\qquad t_a/4\le t<0,
\end{equation*}
with $C$ independent of $t$. Letting $t\uparrow0$ gives a contradiction. Hence $M_0\setminus\{0\}$ is embedded with one sheet.

Equations~\eqref{eq:incomingpower}--\eqref{eq:incominghessian} give local $C^2$ rescaled convergence to $\Sigma$. Interior estimates for MCF give smooth rescaled convergence. Nonflatness makes the origin singular, and every nonzero point is regular. For each sufficiently small fixed $r>0$ and $t<0$ sufficiently close to zero, the local graph representation \eqref{eq:incominggraph} and \eqref{eq:incomingbox} show that $M_t\cap\overline B_r$ is a normal graph over a connected compact truncation of $\Sigma_t$, with the cutting sphere transverse to every conical end. The annular curvature estimates and the curvature bound on $|x|\le\sqrt{|t|}$ give $|A|\le C/(\sqrt{|t|}+|x|)$. The same graph estimates, this curvature bound, and the tubular neighborhoods of $\Sigma_t$ show that the normal map is an embedding for $|z|<c(\sqrt{|t|}+|x|)$ in a fixed ball centered at zero. Huisken's monotonicity formula \cite{Huisken}, applied to the smooth closed initial surface, gives $\cH^n(M_t\cap B_R)\le CR^n$. If the chosen closed realization has other components, retain the component containing the connected truncated shrinker.

\section{Continuation through the singularity}\label{sec:weak}

Let $M_t^-$ be the incoming flow from Section~\ref{sec:incoming}, with limit $M_0$. Choose $r>0$ sufficiently small so that $\partial B_r$ is transverse to $M_0\setminus\{0\}$. Then every component of $M_0\setminus B_r$ meets the cutting sphere. Otherwise, smooth convergence with multiplicity one on a component without boundary would give a closed component of $M_t^-$ disjoint from $M_t^-\cap B_r$, contradicting connectedness. Theorem~\ref{thm:forward} supplies an outgoing flow $M_t^+$ from $M_0$ with the chosen expander $E$. In its construction, take the gluing radius $r_0$ so small that $8r_0<r$. The conical annuli join the connected truncated expander to every component of $M_0\setminus B_r$, so $N_t$ is connected for sufficiently small $t>0$. Since $M_t^+$ is a normal graph over $N_t$, it is also connected. Set
\begin{equation*}
 M_t=\begin{cases}
 M_t^-,&-1\le t<0,\\
 M_0,&t=0,\\
 M_t^+,&0<t\le T_E.
 \end{cases}
\end{equation*}
The family is smooth away from $(0,0)$ by the smooth extension of both flows to the regular part of $M_0$.

\subsection{Weak continuation}

Let $\mu_t=\cH^n\lfloor M_t$ and $\mu_0=\cH^n\lfloor(M_0\setminus\{0\})$. Both sides have $\mu_t(B_R)\le C(R^n+|t|^{n/2})$. Smooth convergence with multiplicity one off zero and this estimate imply varifold convergence and convergence of total area:
\begin{equation}\label{eq:masscontinuity}
 \mu_t\to\mu_0,\qquad
 A(t):=\mu_t(\R^{n+1})\longrightarrow A(0).
\end{equation}
The area identities for the smooth closed mean curvature flows on the two sides imply
\begin{equation}\label{eq:energy}
 \int_{-1}^{T_E}\int|\vec H|^2\dd\mu_t\dd t=A(-1)-A(T_E)<\infty.
\end{equation}
The spacetime $L^1$ norm of $\vec H$ is finite by Cauchy--Schwarz. For $-1\le a<b\le T_E$ and smooth compactly supported $\varphi\ge0$, integrate the smooth evolution identity on the portions of $[a,b]$ outside $(-\delta,\delta)$ and let $\delta\downarrow0$. Equations~\eqref{eq:masscontinuity}--\eqref{eq:energy} give
\begin{equation*}
 \mu_b(\varphi_b)-\mu_a(\varphi_a)
 =\int_a^b\int(\partial_t\varphi+\nabla\varphi\cdot\vec H
           -\varphi|\vec H|^2)\dd\mu_t\dd t.
\end{equation*}
Hence $(\mu_t)_{t\in[-1,T_E]}$ is an integral Brakke flow \cite{Brakke} in the integrated formulation of \cite{BK}. Hausdorff continuity identifies its spacetime support with the track of $M_t$, which is therefore a weak set flow \cite{BK}. Since $M_t$ is a smooth closed hypersurface for every $t\ne0$, the Brakke flow is cyclic \cite{WhiteCurrents}.

For $t\ne0$, let $\Omega_t$ be the bounded region enclosed by the connected closed hypersurface $M_t$. Smooth convergence with multiplicity one off zero preserves the parity of transverse crossings of a path from infinity to any $y\notin M_0$ avoiding zero; Hausdorff convergence excludes additional crossings away from $M_0$. Thus the characteristic functions on both sides have the same pointwise limit off $M_0$, denoted by $\chi_{\Omega_0}$. Common compact support gives
\begin{equation}\label{eq:phase}
 \chi_{\Omega_t}\longrightarrow\chi_{\Omega_0}\quad\hbox{in }L^1.
\end{equation}
Uniform area bounds and $L^1$ lower semicontinuity imply that $\Omega_0$ has finite perimeter; see \cite{Maggi}. Off zero its perimeter measure is $\mu_0$, and a finite-perimeter boundary has no mass at a point. With outward orientations, the boundary currents satisfy \cite[Section 2]{BWLS}
\begin{equation}\label{eq:slicecurrents}
 T_t=\partial[\![\Omega_t]\!],\qquad \norm{T_t}=\mu_t,
 \qquad \mathcal F(T_t-T_0)\le |\Omega_t\mathbin\triangle\Omega_0|\to0,
\end{equation}
where $\mathcal F$ is the flat norm. On each side of time zero, the spacetime region $\{(x,t):x\in\Omega_t\}$ has finite perimeter: its spatial and time variations are bounded by $\int A(t)\dd t$ and $\int\!\int|\vec H|\dd\mu_t\dd t$, respectively. The equal $L^1$ traces in \eqref{eq:phase} cancel the boundary terms at time zero, so the two regions join to a set of finite perimeter \cite{Maggi}. Let $\mathcal T$ be its oriented boundary current, with orientation chosen so that its almost-everywhere time slices in the convention of \cite{BWLS} are $T_t$. Current boundary identities are understood relative to the time interval. Its mass is
\begin{equation}\label{eq:trackmass}
 \mathbf M(\mathcal T)=\int_{-1}^{T_E}\int\sqrt{1+|\vec H|^2}\dd\mu_t\dd t<\infty.
\end{equation}
We use the flat-continuous family $T_t$ as the representative of the time slices of $\mathcal T$. By \eqref{eq:trackmass}, restrictions of $\mathcal T$ to time slabs converge in mass as their endpoints converge. Passing to the limit in the almost-everywhere slicing boundary identity \cite{Maggi} therefore gives the same identity at every time, with one-sided limits at $-1$ and $T_E$. Together with $\norm{T_t}=\mu_t$ in \eqref{eq:slicecurrents}, this gives a matching motion in the sense of \cite{BWLS,Ilmanen}.

The backward density at $(0,0)$ is the Gaussian area of $\Sigma$; the area growth estimates control the Gaussian tails when passing from the incoming rescalings to this density. The Gaussian integrals for the shrinking flow $(\sqrt{-t}\,\Sigma)_{t<0}$, based at a regular point, tend to $1$ at small scales and to the Gaussian area of $\Sigma$ at large scales. Huisken's monotonicity formula \cite{Huisken} shows that this area is at least $1$; equality and smoothness at the base point would force the flow to be a static hyperplane. Since $\Sigma$ is nonflat, the density is greater than $1$. Hence the flow is unit regular \cite{BK}.

\subsection{Almost regularity for surfaces}

Let $n=2$, and let $r_{\rm loc}$ denote the local scale defined in \cite[Definition~2.4]{BK}. Recall that $\rho(x,t)=\sqrt{|t|+|x|^2}$.

Near $(0,0)$, the estimates in Section~\ref{sec:incoming} give $|A|\le C\rho^{-1}$ and a normal map that is an embedding for $|z|<c_0\rho$ on the incoming side. The same bounds hold on the outgoing side by \eqref{eq:Ngeometry}, \eqref{eq:Xpoint}, and the smallness of the normal graph. Indeed, in the inner region, \eqref{eq:Xpoint} gives $|v|/\rho+|\nabla v|\le CB_0\eps^{a/2}$ and $\rho|\nabla^2v|\le CB\rho^a\le CB_0$. The bounds \eqref{eq:Ngeometry} for $N_t$ and these estimates give uniform $C^2$ bounds for local graphs of $M_t$ after scaling each neighborhood by the value of $\rho^{-1}$ at its center. On a sufficiently small scale, these bounds make the normal map injective within each graph patch; the tubular neighborhood of $N_t$ and the small $C^1$ graph bound exclude intersections between normals from different parts of the surface. Thus $M_t$ has an embedded tubular neighborhood of width $c_0\rho$ after decreasing $c_0$ if necessary, where $c_0$ may depend on the fixed $B_0$. Away from $(0,0)$ the flow is smooth on a compact spacetime set.

Fix a regular point $(x,t)$ near $(0,0)$, and put $\rho_0=\rho(x,t)$ and $R=c\rho_0$. For sufficiently small $c$, the cylinder $B_{8R}(x)\times([t-R^2,t+R^2]\cap[-1,T_E])$ avoids $(0,0)$, and $\rho$ is comparable to $\rho_0$ throughout it. Decreasing $c$ ensures $|A|\le10^{-2}R^{-1}$ throughout the cylinder.

For each $s$ in this time interval, let $x_s\in M_s$ be the point obtained from $x$ by normal parametrization. The normal velocity bound gives
\begin{equation*}
 |x_s-x|\le CR^2/\rho_0<R/10.
\end{equation*}
For each $s$, the curvature bound gives a graph through $x_s$ over $\overline B^2_{5R}(0)$ with slope at most $1/10$, in coordinates centered at $x$. Any point of $M_s\cap B_{2R}(x)$ outside this graph has a nearest point in the interior of the graph disk, and the connecting segment is normal and has length less than $3R$. For sufficiently small $c$, we have $3R<c_0\rho$ throughout the cylinder, so this contradicts injectivity of the normal map. Restricting the graph, we may therefore write $F_s(z)=(z,w_s(z))$ over $B^2_{3R}(0)$ to represent all of $M_s\cap B_{2R}(x)$, with
\begin{equation*}
 |Dw_s|\le1,\qquad |D^2w_s|\le C/\rho_0.
\end{equation*}
Since $x_s$ lies on the graph, the estimates $|x_s-x|<R/10$ and $|Dw_s|\le1$ give $|w_s|\le CR$ on the disk. Thus, for every unit vector $e$ in the coordinate plane,
\begin{equation*}
 \partial_e^2|F_s(z)|^2=2+2(\partial_e w_s)^2+2w_s\,\partial_e^2w_s\ge 2-CR/\rho_0\ge1
\end{equation*}
after decreasing $c$. The inverse image of $B_{R}(x)$ is therefore convex, so $M_s\cap B_{R}(x)$ is connected. The definition of $r_{\rm loc}$ now gives
\begin{equation}\label{eq:localscale}
 r_{\rm loc}(x,t)\ge c\sqrt{|t|+r^2}
\end{equation}
near $(0,0)$, with a positive lower bound on the compact regular part.

For fixed small $R>0$ and $t\ne0$, the area bound, the ball $B_{\sqrt{|t|}}$, and dyadic annuli give
\begin{equation}\label{eq:localintegral}
 \int_{M_t\cap B_R}(|t|+r^2)^{-1}\dd\mu_t\le C\bigl(1+\max\{0,\log(R^2/|t|)\}\bigr).
\end{equation}
The right side is integrable in time, so \eqref{eq:localscale}--\eqref{eq:localintegral} imply local spacetime integrability of $r_{\rm loc}^{-2}$. Almost every slice is smooth, the genus is constant on each smooth half of the flow, and the Brakke flow is unit regular. Thus the flow is almost regular by \cite[Definition~2.6]{BK}.

Now we consider the change of genus through the singularity.

Choose a small $r>0$ with $2r<r_*$ such that $\partial B_r$ is transverse to $M_0\setminus\{0\}$, and set $X=M_0\setminus B_r$. For $t_-<0<t_+$ sufficiently close to zero, smooth convergence up to $\partial B_r$ identifies $M_{t_-}^-\setminus B_r$ and $M_{t_+}^+\setminus B_r$ diffeomorphically with $X$.

For the outgoing flow, fix a sufficiently large $R\ge6R_E$. By the cutoff formula, \eqref{eq:trace}, and \eqref{eq:eheight}, the part of $N_t$ between radii $R\sqrt t$ and $2r$ is a graph over $\cC$, with its height divided by the radius and its gradient bounded by $C(r^{\alpha_0}+R^{-2})$. On this region, \eqref{eq:Xpoint} gives $|v|/|x|+|\nabla v|\le CB R^{a-1}t^{a/2}$, which tends to zero as $t\downarrow0$. Taking $r$ small, $R$ large, and then $t$ small, the part of $M_t^+$ between radii $2R\sqrt t$ and $r$ is therefore a small normal graph over $\cC$ and consists of annuli transverse to the concentric spheres. Since $E$ is transverse to $\partial B_{2R}$, the small $C^1$ graph bounds in \eqref{eq:hjets} and \eqref{eq:Xpoint} show that $M_t^+\cap\overline B_{2R\sqrt t}$ is diffeomorphic to $E\cap\overline B_{2R}$ for sufficiently small $t$. Attaching these annuli shows that $M_t^+\cap\overline B_r$ has the same diffeomorphism type. The incoming identification with a compact truncation of $\Sigma$ follows in the same way from the global graph representation \eqref{eq:incominggraph} and the bound \eqref{eq:incomingpower}.

These truncations are connected and have the same number of boundary circles, determined by $\cC\cap S^2$. The common exterior contribution $\chi(X)$ cancels, even when $X$ is disconnected. Additivity of Euler characteristic therefore gives
\begin{equation*}
 \chi(M_{t_+}^+)-\chi(M_{t_-}^-)
 =-2\bigl(\genus(E)-\genus(\Sigma)\bigr).
\end{equation*}
Since both closed surfaces are connected, this gives \eqref{eq:genusjump}.

This proves Theorem~\ref{thm:twosided}.

\section{Riemannian ambient manifolds}\label{sec:riemannian}

\begin{proof}[Proof of Theorem~\ref{thm:riemannian}]
We describe only the changes to Sections~\ref{sec:forward}--\ref{sec:weak}. Use the same weights as before, with $r=|x|=d_g(\mathfrak p,x)$ in $U$ and $r^2$ extended smoothly and positively outside $U$ on a fixed relatively compact neighborhood of $M_0$. On $N_t$ and $S_{r_0}$, define $\cX_T,\cY_T$ using $g$-geodesic coordinates and $g$-normal parametrizations. On $E_t$, retain $r=|y|$, $\rho(y,t)=\sqrt{t+|y|^2}$, and the Euclidean norms and operators. Constants may depend on $g$ and the weight extension on this neighborhood, but remain uniform in sufficiently small $r_0,T$.

Construct $N_t$ by the same cutoff formula in $U$, and use normal exponential graphs $\exp_x^g(v\mathbf n)$, where $\mathbf n$ is a local $g$-unit normal. Heights on one-sided components are normal line bundle sections; pullbacks identify local unit normal frames.

As in \cite[Section 2.3 and Appendix B.1]{CLSZ}, the additional force in normal coordinates is
\begin{equation*}
 \phi_0[X,t]=\langle\vec H_g(X)-\vec H_{\mathrm{euc}}(X),
                         \mathbf n_X\rangle_{\mathrm{euc}},
\end{equation*}
where $\mathbf n_X$ is a local Euclidean unit normal. For $X$ parametrizing $N_t$, the metric coefficients in normal coordinates differ from the Euclidean ones by $O(r^2)$, and their first coordinate derivatives are $O(r)$. Together with \eqref{eq:Ngeometry}, these bounds give
\begin{equation}\label{eq:riemannianforcing}
 |\nabla^j\partial_t^m\phi_0[X,t]|\le C_{jm}\rho^{1-j-2m}\quad(r<4r_0),\qquad
 \norm{\phi_0[X,t]}_{\cY_T(r<4r_0)}\le Cr_0^{2-a}.
\end{equation}
Indeed, the mean-curvature difference has terms of size $O(r^2)|A_N|+O(r)\le C\rho$. Dividing by $\rho^p$ gives $\rho^{1-p}=\rho^{2-a}\le Cr_0^{2-a}$, since $p=a-1$ and $\rho\le Cr_0$ on this region for $T\le r_0^2$.
Let $\mathbf n_{N_t,g}$ be the corresponding $g$-unit normal. Then
\begin{equation*}
 F_N^g:=\langle\vec H_g-\partial_tX,\mathbf n_{N_t,g}\rangle_g
 =\frac{F_N+\phi_0[X,t]}
 {\langle\mathbf n_{N_t,g},\mathbf n_{N_t}\rangle_{\mathrm{euc}}}.
\end{equation*}
Since the denominator is $1+O(r^2)$ and $|F_N|\le C\rho^{-1}$, the difference $F_N^g-F_N$ is bounded by $C(\rho+r^2/\rho)\le C\rho$. Hence $\norm{F_N^g-F_N}_{\cY_T(r<4r_0)}\le Cr_0^{2-a}$. The linearized operator is
\begin{equation*}
 P_N^g=\partial_t-\Delta_{N_t,g}-|A_g|^2
       -\operatorname{Ric}_g(\mathbf n_{N_t,g},\mathbf n_{N_t,g}).
\end{equation*}
On $r<4r_0$, the coefficient differences at orders two, one, and zero have sizes $O(r^2)$, $O(r+r^2/\rho)$, and $O(r^2/\rho^2+r/\rho+1)$, respectively; the last bound includes the Ricci term. Since $r\le\rho$, this gives
\begin{equation}\label{eq:riemanniancoeff}
 P_N^g-P_N=O(\rho^2)\nabla^2+O(\rho)\nabla+O(1).
\end{equation}
The time derivatives in $P_N^g$ and $P_N$ use the $g$-normal and Euclidean normal parametrizations, respectively. Their velocity vectors differ by a Euclidean tangential vector bounded by $Cr^2|V_N|\le Cr^2/\rho$. Changing the normal parametrization therefore adds only a first-order term already included above. For $\norm v_{\cX_T}\le B$, \eqref{eq:Xpoint} gives
\begin{equation*}
 \frac{|(P_N^g-P_N)v|}{\rho^p}
 \le CB(\rho^2+\rho\sqrt t+t)\le CB r_0^2
\end{equation*}
on the inner cutoff support. In the differentiated geometric estimates, each spatial derivative introduces at most a factor $\rho^{-1}$ and each time derivative a factor $\rho^{-2}$. Rescaling multiplies these derivatives by $\sqrt{t_0}$ and $t_0$, respectively. Since $\sqrt{t_0}\le\rho(x_0,t_0)$, the force and coefficient bounds also hold in the scaled $C_P^{0,\alpha}$ norms. Thus the operator error is bounded by $Cr_0^2$ from $\cX_T$ to $\cY_T$ on this support.

On $S_{r_0}$ replace $P_S$ by
\begin{equation*}
 P_S^g=\partial_t-\Delta_{S_{r_0},g}-|A_{S_{r_0},g}|^2
            -\operatorname{Ric}_g(\mathbf n_{S_{r_0},g},\mathbf n_{S_{r_0},g}).
\end{equation*}
Bounded ambient geometry on the fixed compact neighborhood preserves the barrier $2t\widetilde\rho^p$ and Proposition~\ref{prop:outerinverse}; for normal line bundle sections, compare the norm of the solution. Regraphing by $\exp^g$ preserves \eqref{eq:kjets1}--\eqref{eq:kjets2} and \eqref{eq:outerdefect}. The Gauss lemma gives tangential velocity $O(|\partial_tk||\nabla k|)=O(t/r^3)$, so \eqref{eq:outercoeff} and its scaled H\"older bounds persist. Time derivatives retain the small factor after multiplication by $t_0\asymp t$.

Taking $R_o$ to be the zero-initial-data inverse of $P_S^g$, the same cutoff construction gives
\begin{equation}\label{eq:riemannianparametrix}
 \norm{P_N^gQ-\Id}_{\cY_T\to\cY_T}
 \le CC_i(r_0^{\alpha_0}+r_0^2)
       +C(C_i+C_o)(\sqrt\eps+\eps),
 \qquad T=\eps r_0^2.
\end{equation}
The source bound remains $\norm{F_N^g}_{\cY_T}\le C_Fr_0^{-a}$. For exponential graphs, the normal-velocity variation is $\partial_t v$ and the mean-curvature variation is the Jacobi operator. Uniform scaled ambient curvature bounds therefore preserve \eqref{eq:nonlinearLip}. Choose $\eps$ as before and decrease $r_0$ further so that $CC_i(r_0^{\alpha_0}+r_0^2)<1/4$; the same contraction argument applies. The remaining conclusions follow from the proof of Theorem~\ref{thm:forward} using the graphical MCF equation for $g$. The preceding graph estimates and the smooth convergence of the rescaled metrics to the Euclidean metric give \eqref{eq:riemannianmodel}.

Choose the incoming flow from \cite[Theorem 1.4 and Proposition 2.15]{CLSZ}. Its pointwise bounds and exterior estimates allow the cutoff and connectedness argument of Section~\ref{sec:incoming}, so \eqref{eq:incomingpower}--\eqref{eq:incominghessian} apply in a fixed inner ball. On $r\asymp R$, write the incoming flow and $\Sigma_t$ as graphs $u,v$ over a common plane. For this local application of pseudolocality, extend the coordinate metric to a complete bounded-geometry metric as in \cite[Appendix B.1]{CLSZ}, leaving it unchanged near the incoming flow. The pseudolocality theorem of Bing-Long Chen and Le Yin \cite[Theorem 1.4]{CY}, followed by interior derivative estimates, gives uniform scaled bounds on smaller patches after a fixed fraction of $R^2$. The difference $w=u-v$ satisfies
\begin{equation*}
 w_t=a_g^{ij}(y,u,Du)w_{ij}+B^iw_i+Cw
       +\sqrt{1+|Dv|^2}\,\phi_0[(y,v),t].
\end{equation*}
This equation is uniformly parabolic, with uniformly controlled coefficients in the rescaled coordinates. The force $\phi_0$ is evaluated on the graph $y\mapsto(y,v(y,t))$ with the upward Euclidean unit normal. The model and metric estimates bound the derivatives of the last term by $C_{jm}R^{1-j-2m}$. With $\norm w_\infty\le CR^{1+\alpha_0}$ and $\alpha_0=2\lambda_*<1$, linear interior estimates yield
\begin{equation*}
 |D^j\partial_t^mw|
 \le C_{jm}(R^{1+\alpha_0}+R^3)R^{-j-2m}
 \le C_{jm}R^{1+\alpha_0-j-2m}.
\end{equation*}
The time-zero limit therefore satisfies \eqref{eq:trace}. Embeddedness follows from the same Harnack argument on actual sheets, allowing a bounded zeroth-order term.

For the continuation through $t=0$, the maximum principle for $r^2$ keeps the incoming flow in one coordinate ball, as in \cite[Appendix B.1]{CLSZ}; Hausdorff convergence keeps the outgoing flow there for small $T_E$. With $\mu_t=\cH_g^n\lfloor M_t$, AC area growth and the graph estimates give $\mu_t(B_R^g(\mathfrak p))\le C(R^n+|t|^{n/2})$. This bound and smooth convergence away from $\mathfrak p$ give continuity of total $g$-area at zero. On each smooth half of the flow,
\begin{equation*}
 \frac{d}{dt}\operatorname{Area}_g(M_t)
       =-\int_{M_t}|\vec H_g|_g^2\dd\mu_t.
\end{equation*}
Thus $|\vec H_g|_g^2$ is integrable in spacetime, and the localized evolution identity passes through zero as in Section~\ref{sec:weak}, giving an integral Brakke flow. The arguments for the enclosed regions and currents apply in this coordinate ball with $g$-volume and the coordinate orientation, giving a matching motion. By \cite{HW}, the spacetime support of the Brakke flow is a weak set flow. Since the rescaled metrics converge smoothly to the Euclidean metric, the singular density is the Gaussian area of $\Sigma$, greater than one. All other points are regular, so the flow is unit regular.

For $n=2$, the graph bounds give $|A_g|\le C\rho^{-1}$ and an embedded tubular neighborhood of width $c\rho$. As in Section~\ref{sec:weak}, this neighborhood excludes additional sheets in the nearest-point argument, now using $g$-normal geodesics. In normal coordinates centered at $x$, the Hessian of the squared distance from $x$ in the planar graph parameters is bounded below by $(2-CR/\rho(x,t)-CR^2)I$, positive for $R=c\rho(x,t)$ with small $c$. Hence geodesic-ball intersections are connected, and \eqref{eq:localscale}--\eqref{eq:localintegral} give almost regularity. The genus formula follows from the same Euler characteristic calculation with the common exterior.
\end{proof}

\bibliographystyle{amsalpha}
\bibliography{main}
\end{document}